\documentclass[11pt]{amsart}

\usepackage{amssymb}
\usepackage{mathtools}
\usepackage[numbers,sort&compress]{natbib}
\usepackage{xcolor}
\usepackage[margin=1.36in]{geometry}
\usepackage{hyperref}

\hypersetup{
  hidelinks,
  pdfauthor={Feng Guo},
  pdftitle={Palm Disintegration and Height-Profile Recovery for Projected Hardy--Szego Zeros},
  pdfsubject={Palm disintegration, local spacing, and height-profile recovery for projected Hardy--Szego zeros},
  pdfkeywords={Hardy--Szeg\H{o} zero process; determinantal point process;
Palm distribution; projected point process; inverse problem}
}

\title[Palm Disintegration and Height-Profile Recovery]
{Palm Disintegration and Height-Profile Recovery for Projected Hardy--Szeg\H{o} Zeros}

\author{Feng Guo}
\address{School of Mathematics, South China University of Technology,
Guangzhou 510640, P. R. China}
\email{70207994@nuaa.edu.cn}


\numberwithin{equation}{section}

\theoremstyle{plain}
\newtheorem{theorem}{Theorem}[section]
\newtheorem{lemma}[theorem]{Lemma}
\newtheorem{proposition}[theorem]{Proposition}
\newtheorem{corollary}[theorem]{Corollary}

\theoremstyle{definition}

\theoremstyle{remark}
\newtheorem{remark}[theorem]{Remark}

\begin{document}

\begin{abstract}
We study the horizontal projection of the upper-half-plane Hardy--Szeg\H{o} zero process after independent height-dependent thinning, with height retained as an unobserved mark.
We identify the reduced Palm law of the projected process by disintegrating over the missing height coordinate, describe the conditional law of the two heights associated with a pair of projected points and its near- and far-separation limits, and obtain the small-spacing asymptotic for the right nearest neighbour.
We then study recovery of the height profile from the projected second-order structure.
The resulting covariance transform uniquely determines every compactly supported finite positive height measure and, for separated finite unions of intervals, yields uniform Lipschitz recovery of all endpoints, even when the number of intervals is unknown but bounded.
\end{abstract}

\maketitle

\section{Introduction and main results}
\label{sec:introduction-main-results}

Peres and Vir\'ag proved that the zeros of the parameter-one hyperbolic Gaussian analytic function form a conformally invariant determinantal point process \cite[Theorem~1 and Corollary~4]{PeresVirag2005}; see \cite[Chapters~2--5]{HoughKrishnapurPeresVirag2009} for a systematic treatment.
Under a Cayley transformation, its upper-half-plane realization, which we refer to as the Hardy--Szeg\H{o} zero process, is stationary under horizontal translations and has Hermitian correlation kernel
\begin{equation}
\label{eq:hardy-szego-kernel}
K_{\mathbb H}(z,\zeta)
=
-\frac{1}{\pi(z-\overline\zeta)^2},
\qquad z,\zeta\in\mathbb H,
\end{equation}
with respect to planar Lebesgue measure.

We independently retain each point \(x+iy\) with a probability depending only on its height and then project the retained configuration by \(p(x+iy)=x\), so that the height coordinate is unobserved after projection.
This leads to two complementary questions.
The first is probabilistic: how are the unobserved height marks distributed under the reduced Palm law of the projected process, and how do they influence local spacings?
The second is an inverse problem: to what extent does the projected second-order structure determine the height profile?

The determinantal structure is used on the planar process before the height coordinates are integrated out.
No one-dimensional determinantal representation of the projected process is asserted.

\subsection{Main results}
\label{subsec:introduction-main-results}

Fix \(0<\alpha<\beta<\infty\), and define the class of compactly supported retention profiles
\[
\mathcal{W}_{\alpha,\beta}
=
\left\{
w:(0,\infty)\to[0,1]:
w \text{ is Borel},\
w(y)=0 \text{ for }y\notin[\alpha,\beta],\
\int_{\alpha}^{\beta}w(y)\,\mathrm{d}y>0
\right\}.
\]
Given \(w\in\mathcal{W}_{\alpha,\beta}\), we retain each zero \(z=x+i y\) independently with probability \(w(y)\).
Identifying \(\mathbb H\) with \(\mathbb R\times(0,\infty)\), we regard the second coordinate as the height mark and denote the retained marked point process by \(M_w\).
The process \(M_w\) is determinantal on \(\mathbb H\) with kernel
\[
K_w(z,\zeta)
=
\sqrt{w(\operatorname{Im}z)}
K_{\mathbb H}(z,\zeta)
\sqrt{w(\operatorname{Im}\zeta)},
\]
as recorded in Lemma~\ref{L:height-thinning-kernel}.

\paragraph{Projected process and second-order structure.}
Let \(p:\mathbb H\to\mathbb R\), \(p(x+iy)=x\), and define the projected process \(X_w=p_\#M_w\).
The process \(X_w\) is stationary, locally finite, and simple, with intensity
\[
\lambda_w
=
\frac{1}{4\pi}\int_0^\infty \frac{w(y)}{y^2}\,\mathrm dy;
\]
see Lemmas~\ref{L:projected-local-finiteness-stationarity} and \ref{L:projected-simplicity}.
Its factorial moment densities are obtained by integrating the planar determinantal densities over the height coordinates and admit bounded continuous representatives in every order; see Proposition~\ref{P:projected-factorial-intensities}.
Throughout, we use these representatives.
By stationarity, write \(\rho_w^{(2)}(r)=\rho_w^{(2)}(0,r)\) and define the second factorial cumulant density by \(g_w(r)=\rho_w^{(2)}(r)-\lambda_w^2\).
The key second-order identity, proved in Proposition \ref{P:projected-covariance-transform}, is
\[
g_w(r)
=
-\frac{1}{\pi^2}
\int_0^\infty\int_0^\infty
\frac{w(y)w(v)}
{\left(r^2+(y+v)^2\right)^2}
\,\mathrm dy\,\mathrm dv.
\]
Moreover, \(\rho_w^{(2)}(r)=\lambda_w^2+g_w(r)>0\) for every \(r\in\mathbb R\).

\paragraph{Palm disintegration and local spacing.}

We next identify the projected reduced Palm distribution.
Let \(\mathcal N_{\mathbb R}\) denote the space of locally finite simple counting measures on \(\mathbb R\), equipped with the Borel \(\sigma\)-field induced by the vague topology.
Write \(\mathbb P_w^{0,!}\) for the reduced Palm distribution of \(X_w\) at the origin and \(\mathbb E_w^{0,!}\) for its expectation.
Section~\ref{sec:projected-palm-disintegration} constructs a weakly continuous height-indexed version \(\mathsf P_{w,y}^{!}\) of the marked reduced Palm law of \(M_w\) at \(iy\); expectation with respect to \(\mathsf P_{w,y}^{!}\) is denoted by \(\mathbb E_{M_w}^{iy,!}\).
Define the probability measure
\[
\nu_w(\mathrm dy)
=
\frac{w(y)y^{-2}\,\mathrm dy}
{\displaystyle\int_0^\infty w(u)u^{-2}\,\mathrm du}.
\]

\begin{theorem}[Projected Palm disintegration]
\label{T:projected-palm-disintegration}
For every bounded measurable functional \(F:\mathcal N_{\mathbb R}\to\mathbb R\),
\[
\mathbb E_w^{0,!}[F(X_w)]
=
\int_0^\infty
\mathbb E_{M_w}^{iy,!}
\left[
F(p_\#M_w)
\right]
\,\nu_w(\mathrm dy).
\]
\end{theorem}

Section~\ref{sec:projected-palm-disintegration} realizes this identity on a joint disintegration space.
On that space, the unobserved height mark of the removed Palm point has distribution \(\nu_w\), and, conditional on height \(y\), the remaining marked configuration has law \(\mathsf P_{w,y}^{!}\).
The height mark is not asserted to be a measurable function of the projected reduced Palm configuration.

The second factorial Campbell measure gives the corresponding two-height description.
For \(r\in\mathbb R\), define the probability measure
\[
\mathsf Q_{w,r}(\mathrm dy_0\,\mathrm dy_1)
=
\frac{w(y_0)w(y_1)}{\rho_w^{(2)}(r)}
\left[
\frac{1}{16\pi^2y_0^2y_1^2}
-
\frac{1}{\pi^2\left(r^2+(y_0+y_1)^2\right)^2}
\right]
\,\mathrm dy_0\,\mathrm dy_1.
\]
The next proposition identifies the conditional law of the two height marks given their horizontal displacement.

\begin{proposition}
\label{P:hidden-two-mark-law}
For every nonnegative measurable function \(H:\mathbb R^2\times(0,\infty)^2\to[0,\infty]\),
\[
\begin{aligned}
&\mathbb E\sum_{z_0,z_1\in M_w}^{\ne}
H\bigl(\operatorname{Re}z_0,
       \operatorname{Re}z_1-\operatorname{Re}z_0,
       \operatorname{Im}z_0,\operatorname{Im}z_1\bigr)
\\
&\quad=
\int_{\mathbb R}\int_{\mathbb R}
\rho_w^{(2)}(r)
\int_{(0,\infty)^2}
H(x,r,y_0,y_1)
\,\mathsf Q_{w,r}(\mathrm dy_0\,\mathrm dy_1)
\,\mathrm dr\,\mathrm dx.
\end{aligned}
\]
Moreover, \(r\mapsto\mathsf Q_{w,r}\) is continuous in total variation and is the unique weakly continuous version of the conditional height kernel in this disintegration.
\end{proposition}

The conditional height kernel has the following two limiting regimes in total variation, as shown in Corollary~\ref{C:far-hidden-height-limit}:
\[
\mathsf Q_{w,r}\longrightarrow\nu_w\otimes\nu_w
\quad\text{as }|r|\to\infty,
\qquad
\mathsf Q_{w,r}\longrightarrow\mathsf Q_{w,0}
\quad\text{as }r\to0.
\]
Thus the height marks of widely separated projected points are asymptotically independent, whereas the density of \(\mathsf Q_{w,0}\) vanishes on the diagonal \(y_0=y_1\).

We next consider the small-spacing behavior of the right nearest-neighbour distance.

\begin{theorem}
\label{T:linear-right-short-gap}
Under \(\mathbb P_w^{0,!}\), let \(S_w:=\inf(\operatorname{supp}X_w\cap(0,\infty))\), with \(\inf\varnothing:=\infty\).
Then, as \(t\downarrow0\),
\[
\mathbb P_w^{0,!}\{S_w\leq t\}
=
\frac{\rho_w^{(2)}(0)}{\lambda_w}\,t+O_w(t^2).
\]
The implicit constant in \(O_w(\,\cdot\,)\) may depend on \(w\).
In particular, the leading coefficient is strictly positive.
\end{theorem}

\paragraph{Identifiability and stable reconstruction.}

We next turn to the deterministic inverse problem.
Let \(\mu\) be a finite positive measure supported in a compact subset of \((0,\infty)\), and define
\begin{equation}
\label{eq:height-measure-covariance-transform}
G_\mu(r)
=
\int_0^\infty\int_0^\infty
\frac{\mu(\mathrm dy)\,\mu(\mathrm dv)}
{\bigl(r^2+(y+v)^2\bigr)^2},
\qquad r>0.
\end{equation}
For the probabilistic profile \(\mu_w(\mathrm dy)=w(y)\,\mathrm dy\), the second factorial cumulant density satisfies \(g_w(r)=-\pi^{-2}G_{\mu_w}(r)\).

\begin{theorem}[Stieltjes--Laplace identifiability]
\label{T:stieltjes-laplace-identifiability}
Let \(\mu\) and \(\widetilde\mu\) be finite positive measures supported in a common compact subset of \((0,\infty)\).
If \(G_\mu(r)=G_{\widetilde\mu}(r)\) on a nonempty open subinterval of \((0,\infty)\), then \(\mu=\widetilde\mu\).
\end{theorem}

For finite-union profiles, let
\[
I(\vartheta)
=
\bigcup_{j=1}^{m}[a_j,b_j],
\qquad
0<a_1<b_1<\cdots<a_m<b_m<\infty,
\]
and set \(\mu_\vartheta(\mathrm dy) =\mathbf 1_{I(\vartheta)}(y)\,\mathrm dy\).
Fix \(m\geq1\), \(0<\alpha<\beta<\infty\), \(\delta>0\), and \(0<r_-<r_+<\infty\), and define
\[
\Theta_{m,\alpha,\beta,\delta}
=
\left\{
\vartheta:
\begin{array}{l}
\alpha\leq a_1,\quad b_m\leq\beta,\\
b_j-a_j\geq\delta\quad (1\leq j\leq m),\\
a_{j+1}-b_j\geq\delta\quad (1\leq j<m)
\end{array}
\right\}.
\]
This compact class is nonempty precisely when
\(
(2m-1)\delta\leq\beta-\alpha.
\)

\begin{theorem}[Lipschitz stability of endpoint recovery]
\label{T:finite-union-stability}
Assume \((2m-1)\delta\leq\beta-\alpha\).
There exists \(C<\infty\), depending only on \(m,\alpha,\beta,\delta,r_-\), and \(r_+\), such that
\[
\|\vartheta-\widetilde\vartheta\|_{\mathbb R^{2m}}
\leq
C
\sup_{r\in[r_-,r_+]}
\left|
G_{\mu_\vartheta}(r)-G_{\mu_{\widetilde\vartheta}}(r)
\right|
\]
for all \(\vartheta,\widetilde\vartheta\in \Theta_{m,\alpha,\beta,\delta}\), where \(\|\cdot\|_{\mathbb R^{2m}}\) denotes the Euclidean norm.
\end{theorem}

The number of intervals can also be recovered stably when it is unknown but bounded.
Corollary~\ref{C:stable-recovery-of-m} shows that the covariance images of distinct admissible parameter classes are separated by a positive distance; once \(m\) is identified, Theorem~\ref{T:finite-union-stability} gives the endpoint estimate.

\subsection{Background and related work}
\label{subsec:introduction-background}

Determinantal point processes originate in the work of Macchi \cite{Macchi1975}; general constructions, correlation functions, and operator criteria are developed in \cite{HoughKrishnapurPeresVirag2006,Soshnikov2000}; for determinantal probability measures on countable spaces, see \cite{Lyons2003}.
The determinantal description of the parameter-one hyperbolic zero process is due to Peres and Vir\'ag \cite{PeresVirag2005}; the Gaussian analytic function and zero-process background is treated systematically in \cite{HoughKrishnapurPeresVirag2009}.
Coordinate projections of determinantal point processes were studied in \cite{MazoyerCoeurjollyAmblard2020}; unlike that bounded-product setting, our noncompact projection is used for Palm, spacing, and inverse questions.

For indicator retention profiles \(w=\mathbf 1_I\), the projected process and its long-window fluctuations are studied in \cite{AiGuoZhou2026}, where the covariance density is computed and Brillinger mixing and a functional central limit theorem are established.
In this case, Proposition~\ref{P:projected-covariance-transform} agrees with the covariance formula in \cite{AiGuoZhou2026}.
The present paper treats general compactly supported Borel retention profiles and addresses the different questions of projected Palm disintegration, local spacing, and recovery of the height profile.

Factorial moment measures and marked point processes are used in the standard framework of \cite{DaleyVereJones2003}.
For Palm distributions, Campbell--Mecke identities, disintegration on configuration spaces, and the standard mark-averaging relation for the reduced Palm law of the ground process, see \cite[Chapter~13]{DaleyVereJones2008} and \cite{Kallenberg2017}; see also \cite{BaccelliBremaud1987} for stationary marked Palm theory.
For determinantal Palm kernels, see \cite{ShiraiTakahashi2003}.
The additional content here is the explicit height mixing law and the weakly continuous determinantal height-indexed Palm version used below.
For independent location-dependent thinning of determinantal processes, see \cite{LavancierMollerRubak2015}.

The inverse argument uses classical uniqueness principles for Stieltjes and Laplace transforms; see \cite{Widder1941}.
For finite-union profiles, multiplying the recovered Laplace transform by \(t\) gives a finite signed exponential sum, linking endpoint recovery to Prony-type inversion.
Prony mappings and their collision singularities are studied in \cite{BatenkovYomdin2014}, while numerical recovery of exponential sums from noisy samples is treated in \cite{PottsTasche2010}.
These works provide context rather than the stability estimate proved here: our Lipschitz stability estimate is proved directly for the covariance map in \(C([r_-,r_+])\).

The inverse results below assume that the projected covariance transform is known exactly.
Estimating that transform from a single realization of the projected process would require additional statistical arguments and is not considered here.

\subsection{Proof strategy}
\label{subsec:introduction-proof-strategy}

Integrating the planar one-point intensity over the height variable gives the projected intensity and local finiteness, while horizontal translation invariance gives stationarity.
Projected simplicity is proved by bounding the expected number of pairs in a bounded interval whose horizontal coordinates differ by less than \(\varepsilon\).
For higher factorial moments, the planar determinantal formula, Fubini's theorem, Hadamard's inequality, and dominated convergence yield the bounded continuous densities in Proposition~\ref{P:projected-factorial-intensities}; the two-point case gives the explicit covariance transform in Proposition~\ref{P:projected-covariance-transform}.

For Theorem~\ref{T:projected-palm-disintegration}, we compare the marked and projected Campbell measures.
The resulting disintegration identifies the mixing measure \(\nu_w\), while the explicit reduced determinantal kernel provides a weakly continuous height-indexed Palm version.
The second factorial Campbell measure yields the conditional two-height kernel in Proposition~\ref{P:hidden-two-mark-law}; total-variation continuity and uniqueness of disintegration determine its weakly continuous version for every horizontal separation.
The first two factorial moments under the projected reduced Palm law, together with the Bonferroni inequalities, then yield Theorem~\ref{T:linear-right-short-gap}.

To prove Theorem~\ref{T:stieltjes-laplace-identifiability}, write \(\nu=\mu\ast\mu\).
Integrating \(G_\mu(\sqrt{u})\) with respect to \(u\) yields the Stieltjes transform of the pushforward of \(\nu\) under \(s\mapsto s^2\).
Stieltjes uniqueness recovers \(\nu\); taking Laplace transforms of \(\nu=\mu\ast\mu\) and using positivity then determines \(\mu\).
On separated finite-union classes, injectivity of the differential and a uniform second-order Taylor estimate give a local lower Lipschitz bound; compactness and injectivity of the covariance map extend it to the global stability estimate in Theorem~\ref{T:finite-union-stability}.

\medskip
\noindent\textbf{Organization of the paper.}
Section \ref{sec:height-selected-projected-process} constructs the process obtained by height-dependent thinning and derives the correlation structure of its projection.
Section \ref{sec:projected-palm-disintegration} proves the Palm disintegration, the conditional height distributions, and the right small-spacing asymptotic.
Section \ref{sec:stieltjes-laplace-identifiability} proves identifiability and stable reconstruction for finite unions.

\section{Projected process and correlation structure}
\label{sec:height-selected-projected-process}

This section introduces height-dependent thinning and establishes local finiteness, stationarity, simplicity, and factorial moment densities for the horizontal projection.

\subsection{Height-dependent thinning and projection}
\label{subsec:model-hardy-szego}

Recall that the upper-half-plane Hardy--Szeg\H{o} zero process \(Z_{\mathbb H}\) is a determinantal point process with Hermitian kernel \eqref{eq:hardy-szego-kernel} relative to planar Lebesgue measure \(\mathrm dA\).
In particular, \(K_{\mathbb H}(x+iy,x+iy)=1/(4\pi y^2)\).
For every integer \(k\ge1\), the \(k\)-th factorial moment measure of \(Z_{\mathbb H}\) has density, with respect to \(\mathrm dA^{\otimes k}\),
\begin{equation}
\label{eq:hardy-szego-factorial-density}
\rho_{\mathbb H}^{(k)}(z_1,\ldots,z_k)
=
\det[K_{\mathbb H}(z_i,z_j)]_{i,j=1}^k.
\end{equation}
Hadamard's inequality gives
\begin{equation}
\label{eq:hardy-szego-hadamard-bound}
0
\le
\rho_{\mathbb H}^{(k)}(z_1,\ldots,z_k)
\le
\prod_{j=1}^k K_{\mathbb H}(z_j,z_j)
=
\prod_{j=1}^k
\frac{1}{4\pi(\operatorname{Im}z_j)^2}.
\end{equation}
The kernel satisfies \(K_{\mathbb H}(z+t,\zeta+t)=K_{\mathbb H}(z,\zeta)\) for \(t\in\mathbb R\), and \(Z_{\mathbb H}\) is horizontally stationary.

\paragraph{Admissible retention profiles.}
\label{subsec:height-selection-class}

Recall that, for fixed \(0<\alpha<\beta<\infty\), \(\mathcal W_{\alpha,\beta}\) denotes the class of Borel functions \(w:(0,\infty)\to[0,1]\) supported in \([\alpha,\beta]\) and satisfying
\[
\int_\alpha^\beta w(y)\,\mathrm dy>0.
\]
For every \(w\in\mathcal W_{\alpha,\beta}\),
\[
0<\int_0^\infty w(y)y^{-2}\,\mathrm dy<\infty.
\]
A compact interval \(I\subset[\alpha,\beta]\) of positive length gives the deterministic height-window profile \(w=\mathbf 1_I\).

\paragraph{Height-dependent thinning.}
\label{subsec:height-thinning}

Fix \(w\in\mathcal W_{\alpha,\beta}\).
Recall that \(M_w\) is obtained from \(Z_{\mathbb H}\) by independently retaining each point \(z=x+iy\) with probability \(w(y)\), with \(y\) regarded as the height mark under the identification \(\mathbb H=\mathbb R\times(0,\infty)\).
For independent thinning of determinantal point processes with spatially varying retention probabilities, see, e.g., \cite{LavancierMollerRubak2015}.
The following lemma identifies the determinantal kernel of \(M_w\).

\begin{lemma}
\label{L:height-thinning-kernel}
The process \(M_w\) is determinantal on \(\mathbb H\), relative to \(\mathrm dA\), with kernel
\[
K_w(z,\zeta)
=
\sqrt{w(\operatorname{Im}z)}
K_{\mathbb H}(z,\zeta)
\sqrt{w(\operatorname{Im}\zeta)}.
\]
In particular, for every integer \(k\ge1\), its \(k\)-th factorial moment density is
\[
\rho_{M_w}^{(k)}(z_1,\ldots,z_k)
=
\det[K_w(z_i,z_j)]_{i,j=1}^k.
\]
\end{lemma}

\begin{proof}
Fix \(k\ge1\), and let \(F\) be a nonnegative compactly supported measurable function on \(\mathbb H^k\).
Independent thinning and \eqref{eq:hardy-szego-factorial-density} give
\[
\begin{aligned}
\mathbb E\,
\smashoperator[r]{\sum_{z_1,\ldots,z_k\in M_w}^{\ne}}
F(z_1,\ldots,z_k)
&=
\int_{\mathbb H^k}
F(z_1,\ldots,z_k)
\Bigl(\prod_{j=1}^k w(\operatorname{Im}z_j)\Bigr)
\det[K_{\mathbb H}(z_i,z_j)]_{i,j=1}^k
\prod_{j=1}^k\mathrm dA(z_j)
\\
&=
\int_{\mathbb H^k}
F(z_1,\ldots,z_k)
\det[K_w(z_i,z_j)]_{i,j=1}^k
\prod_{j=1}^k\mathrm dA(z_j).
\end{aligned}
\]
Hence \(M_w\) is determinantal with the asserted factorial moment densities.
\end{proof}

\subsection{Basic properties and factorial moment densities}
\label{subsec:projection-intensity-stationarity}

Recall the horizontal projection \(p:\mathbb H\to\mathbb R\), \(p(x+iy)=x\), and the projected process
\[
X_w=p_\#M_w
=\sum_{z=x+iy\in M_w}\delta_x.
\]
The following lemma establishes its local finiteness and stationarity and identifies its intensity.

\begin{lemma}
\label{L:projected-local-finiteness-stationarity}
The projected process \(X_w\) is locally finite and stationary on \(\mathbb R\), with intensity measure \(\lambda_w\,\mathrm dx\), where
\begin{equation}
\label{eq:projected-intensity}
\lambda_w
=
\frac{1}{4\pi}
\int_0^\infty\frac{w(y)}{y^2}\,\mathrm dy.
\end{equation}
Equivalently,
\(
\mathbb E X_w(B)=\lambda_w|B|
\)
for every bounded Borel set \(B\subset\mathbb R\).
\end{lemma}

\begin{proof}
By Lemma~\ref{L:height-thinning-kernel} and \eqref{eq:hardy-szego-kernel}, for every bounded Borel set \(B\subset\mathbb R\),
\[
\begin{aligned}
\mathbb E X_w(B)
&=
\int_B\int_0^\infty
K_w(x+iy,x+iy)\,\mathrm dy\,\mathrm dx
\\
&=
\frac{|B|}{4\pi}
\int_0^\infty\frac{w(y)}{y^2}\,\mathrm dy
=
\lambda_w|B|.
\end{aligned}
\]
The defining conditions on \(w\) give \(0<\lambda_w<\infty\).
Hence \(X_w([-n,n])<\infty\) almost surely for every \(n\ge1\).
Taking the intersection of these probability-one events proves local finiteness.

Finally, by \eqref{eq:hardy-szego-kernel} and the dependence of \(w\) only on height, \(K_w(z+t,\zeta+t)=K_w(z,\zeta)\) for \(t\in\mathbb R\).
Hence \(M_w\) is horizontally stationary, and \(X_w\) is stationary.
\end{proof}

The simplicity of \(M_w\) on \(\mathbb H\) does not preclude distinct points from sharing the same real coordinate.
The next lemma rules out such collisions after projection.

\begin{lemma}
\label{L:projected-simplicity}
The projected process \(X_w\) is simple: almost surely, \(X_w(\{x\})\le1\) for every \(x\in\mathbb R\).
\end{lemma}

\begin{proof}
It suffices to rule out projected collisions in every bounded interval with rational endpoints.
Fix such an interval \(B\), and, for \(\varepsilon>0\), define
\[
M_{B,\varepsilon}
=
\sum_{\substack{z,\zeta\in M_w\\z\ne\zeta}}
\mathbf 1_B(\operatorname{Re}z)
\mathbf 1
\left\{
|\operatorname{Re}z-\operatorname{Re}\zeta|<\varepsilon
\right\}.
\]
A projected collision in \(B\) implies \(M_{B,\varepsilon}\ge1\) for every \(\varepsilon>0\).
Hence
\[
\mathbb{P}
\left\{
\exists x\in B:X_w(\{x\})\ge2
\right\}
\le
\mathbb{P}\{M_{B,\varepsilon}\ge1\}
\le
\mathbb{E}M_{B,\varepsilon}.
\]
By Lemma~\ref{L:height-thinning-kernel} and \eqref{eq:hardy-szego-hadamard-bound},
\[
\begin{aligned}
\mathbb{E}M_{B,\varepsilon}
&=
\int_B
\int_{\{|x'-x|<\varepsilon\}}
\int_0^\infty\int_0^\infty
\rho_{M_w}^{(2)}(x+iy,x'+iv)
\,\mathrm dy\,\mathrm dv\,\mathrm dx'\,\mathrm dx
\\
&\le
\int_B
\int_{\{x'\in\mathbb{R}:|x'-x|<\varepsilon\}}
\int_0^\infty
\int_0^\infty
\frac{w(y)}{4\pi y^2}
\frac{w(v)}{4\pi v^2}
\,\mathrm{d}y\,\mathrm{d}v\,\mathrm{d}x'\,\mathrm{d}x
=
2\varepsilon |B|\lambda_w^2.
\end{aligned}
\]
Letting \(\varepsilon\downarrow0\) shows that the collision probability in \(B\) is zero.
A countable exhaustion of \(\mathbb R\) by bounded intervals with rational endpoints proves the claim.
\end{proof}

\paragraph{Factorial moment densities.}
\label{subsec:projected-factorial-intensities}

Factorial sums are ordered: \(\sum_{x_1,\ldots,x_k\in X_w}^{\neq}\) denotes summation over ordered \(k\)-tuples of pairwise distinct projected points.

\begin{proposition}
\label{P:projected-factorial-intensities}
For every integer \(k\ge1\), the \(k\)-th factorial moment measure of \(X_w\) is absolutely continuous with respect to Lebesgue measure on \(\mathbb R^k\).
A bounded continuous, symmetric, and diagonally translation-invariant representative of its density is given by
\[
\rho_w^{(k)}(x_1,\ldots,x_k)
=
\int_0^\infty\cdots\int_0^\infty
\left(\prod_{j=1}^k w(y_j)\right)
\det
\left[
K_{\mathbb H}(x_i+iy_i,x_j+iy_j)
\right]_{i,j=1}^k
\,\mathrm dy_1\cdots\mathrm dy_k.
\]
Moreover,
\[
0
\le
\rho_w^{(k)}(x_1,\ldots,x_k)
\le
\lambda_w^k,
\qquad
(x_1,\ldots,x_k)\in\mathbb R^k.
\]
\end{proposition}

\begin{proof}
Fix \(k\ge1\), and let \(F\) be a nonnegative compactly supported measurable function on \(\mathbb R^k\).
By Lemma~\ref{L:projected-simplicity} and the factorial moment formula in Lemma~\ref{L:height-thinning-kernel},
\[
\begin{aligned}
\mathbb E
\sum_{x_1,\ldots,x_k\in X_w}^{\neq}
F(x_1,\ldots,x_k)
&=
\mathbb E
\sum_{z_1,\ldots,z_k\in M_w}^{\neq}
F(\operatorname{Re}z_1,\ldots,\operatorname{Re}z_k)
\\
&=
\int_{\mathbb R^k}
F(x_1,\ldots,x_k)
\rho_w^{(k)}(x_1,\ldots,x_k)
\,\mathrm dx_1\cdots\mathrm dx_k.
\end{aligned}
\]

By \eqref{eq:hardy-szego-hadamard-bound} and \eqref{eq:projected-intensity},
\[
0
\le
\rho_w^{(k)}(x_1,\ldots,x_k)
\le
\int_0^\infty\cdots\int_0^\infty
\prod_{j=1}^k
\frac{w(y_j)}{4\pi y_j^2}
\,\mathrm dy_1\cdots\mathrm dy_k
=
\lambda_w^k.
\]
For fixed heights, the integrand defining \(\rho_w^{(k)}\) is continuous in \((x_1,\ldots,x_k)\), while the support condition on \(w\) and \eqref{eq:hardy-szego-hadamard-bound} yield an integrable majorant independent of these coordinates.
Dominated convergence gives continuity.
Relabeling the height variables and simultaneously permuting the corresponding rows and columns gives symmetry, while \eqref{eq:hardy-szego-kernel} gives diagonal translation invariance.
The integral formula fixes the continuous representative of \(\rho_w^{(k)}\) used below; changing its values on coordinate diagonals would not affect the factorial moment measure.
\end{proof}

The densities in Proposition~\ref{P:projected-factorial-intensities} are obtained by integrating the marked planar determinantal densities over the height variables; no determinantal representation of \(X_w\) on \(\mathbb R\) is asserted.

\subsection{Second factorial cumulant density}
\label{subsec:two-point-function}

For \(k=2\), diagonal translation invariance implies that the factorial moment density depends only on \(r=x_2-x_1\); write
\[
\rho_w^{(2)}(r)=\rho_w^{(2)}(0,r),
\qquad
g_w(r)=\rho_w^{(2)}(r)-\lambda_w^2.
\]
The two-point density and its factorial cumulant admit the following explicit formulas.

\begin{proposition}[Projected covariance transform]
\label{P:projected-covariance-transform}
The stationary two-point factorial moment density admits the continuous representative
\begin{equation}
\label{eq:projected-two-point-density}
\rho_w^{(2)}(r)
=
\lambda_w^2
-
\frac{1}{\pi^2}
\int_0^\infty\int_0^\infty
\frac{w(y)w(v)}
{\left(r^2+(y+v)^2\right)^2}
\,\mathrm dy\,\mathrm dv,
\qquad r\in\mathbb R.
\end{equation}
Equivalently,
\[
g_w(r)
=
-\frac{1}{\pi^2}
\int_0^\infty\int_0^\infty
\frac{w(y)w(v)}
{\left(r^2+(y+v)^2\right)^2}
\,\mathrm dy\,\mathrm dv.
\]
Moreover, \(\rho_w^{(2)}(r)>0\) for every \(r\in\mathbb R\).
\end{proposition}

\begin{proof}
By Proposition~\ref{P:projected-factorial-intensities} and \eqref{eq:hardy-szego-kernel}, for \(r\ne0\),
\[
\begin{aligned}
\rho_w^{(2)}(r)
&=
\int_0^\infty\int_0^\infty
w(y)w(v)
\left[
\frac{1}{16\pi^2y^2v^2}
-
\frac{1}{\pi^2\left(r^2+(y+v)^2\right)^2}
\right]
\,\mathrm dy\,\mathrm dv
\\
&=
\lambda_w^2
-
\frac{1}{\pi^2}
\int_0^\infty\int_0^\infty
\frac{w(y)w(v)}
{\left(r^2+(y+v)^2\right)^2}
\,\mathrm dy\,\mathrm dv.
\end{aligned}
\]
On the support of \(w(y)w(v)\), \(y+v\ge2\alpha\), so the integral term is finite and continuous in \(r\) by dominated convergence.
The identity therefore extends to \(r=0\) and gives the asserted continuous representative.
The formula for \(g_w\) follows from its definition.

For every \(r\in\mathbb R\),
\[
\frac{1}{16\pi^2y^2v^2}
-
\frac{1}{\pi^2\left(r^2+(y+v)^2\right)^2}
=
\frac{\left(r^2+(y-v)^2\right)
\left(r^2+y^2+6yv+v^2\right)}
{16\pi^2y^2v^2\left(r^2+(y+v)^2\right)^2}.
\]
This factor is strictly positive for \(r\ne0\) and, for \(r=0\), vanishes only when \(y=v\).
Since
\(
E=\{y\in[\alpha,\beta]:w(y)>0\}
\)
has positive Lebesgue measure and the diagonal has zero measure in \(E\times E\), integration against \(w(y)w(v)\,\mathrm dy\,\mathrm dv\) gives \(\rho_w^{(2)}(r)>0\) for every \(r\in\mathbb R\).
\end{proof}

\section{Palm disintegration and local spacing}
\label{sec:projected-palm-disintegration}

This section proves Theorem~\ref{T:projected-palm-disintegration}, Proposition~\ref{P:hidden-two-mark-law}, and Theorem~\ref{T:linear-right-short-gap}.
It also derives conditional distributions of the unobserved height marks and factorial moment identities under the projected reduced Palm law.
The projected and marked reduced Palm laws are defined through the Campbell identities below.

Fix \(w\in\mathcal W_{\alpha,\beta}\) throughout, and recall that \(X_w=p_\#M_w\).
By Lemmas~\ref{L:projected-local-finiteness-stationarity} and \ref{L:projected-simplicity}, \(X_w\) is locally finite, stationary, and simple, with intensity \(\lambda_w\) given by \eqref{eq:projected-intensity}.

\subsection{Projected reduced Palm disintegration}
\label{subsec:palm-definitions}

Let \(\mathcal N_{\mathbb R}\) and \(\mathcal N_{\mathbb H}\) denote the spaces of locally finite simple counting measures on \(\mathbb R\) and \(\mathbb H\), endowed with the Borel \(\sigma\)-fields induced by the vague topology.
For \(x\in\mathbb R\), let \(\theta_x\) denote horizontal translation by \(-x\) on either space; then \(p_\#(\theta_x\xi)=\theta_x(p_\#\xi)\).
We use the Palm-kernel and reduced Palm conventions of \cite[Chapter~6]{Kallenberg2017} and \cite[Palm Theory, pp.~268--354]{DaleyVereJones2008}; see also \cite[pp.~2--18]{BaccelliBremaud1987} for the stationary marked setting.

The reduced Palm distribution \(\mathbb P_w^{0,!}\) of \(X_w\) at the origin is characterized by
\begin{equation}
\label{eq:projected-reduced-palm-campbell}
\mathbb E
\sum_{x\in X_w}
\mathbf 1_B(x)
F\left(\theta_x(X_w-\delta_x)\right)
=
|B|\lambda_w
\mathbb E_w^{0,!}[F(X_w)]
\end{equation}
for every bounded Borel set \(B\subset\mathbb R\) with \(|B|>0\) and every bounded measurable \(F:\mathcal N_{\mathbb R}\to\mathbb R\).

A reduced Palm probability kernel \(z\mapsto\mathbb P_{M_w}^{z,!}\), defined up to \(K_w(z,z)\,\mathrm dA(z)\)-null sets, is characterized by
\begin{equation}
\label{eq:marked-reduced-palm-campbell}
\mathbb E
\sum_{z\in M_w}
H(z,M_w-\delta_z)
=
\int_{\mathbb H}
\int_{\mathcal N_{\mathbb H}}
H(z,\eta)\,
\mathbb P_{M_w}^{z,!}(\mathrm d\eta)\,
K_w(z,z)\,\mathrm dA(z)
\end{equation}
for every nonnegative measurable \(H:\mathbb H\times\mathcal N_{\mathbb H}\to[0,\infty]\).

Recall the probability measure
\begin{equation}
\label{eq:palm-height-mixing-measure}
\nu_w(\mathrm dy)
=
\frac{w(y)y^{-2}\,\mathrm dy}
{\int_0^\infty w(u)u^{-2}\,\mathrm du}.
\end{equation}
For \(z=x+iy\),
\begin{equation}
\label{eq:marked-campbell-intensity}
K_w(z,z)\,\mathrm dA(z)
=
\frac{w(y)}{4\pi y^2}\,\mathrm dx\,\mathrm dy
=
\lambda_w\,\mathrm dx\,\nu_w(\mathrm dy).
\end{equation}

The next proposition constructs the height-indexed Palm family used in the disintegration.

\begin{proposition}
\label{P:canonical-height-indexed-palm-family}
For \(y>0\), define
\[
\begin{aligned}
K_{\mathbb H}^{iy,!}(z,\zeta)
&=
K_{\mathbb H}(z,\zeta)
-
\frac{K_{\mathbb H}(z,iy)K_{\mathbb H}(iy,\zeta)}
     {K_{\mathbb H}(iy,iy)},\\
\widehat K_{w,y}^{!}(z,\zeta)
&=
\sqrt{w(\operatorname{Im}z)}
K_{\mathbb H}^{iy,!}(z,\zeta)
\sqrt{w(\operatorname{Im}\zeta)}.
\end{aligned}
\]
Then the following hold.
\begin{enumerate}
\item[\textup{(i)}]
For every \(y>0\), the kernel \(\widehat K_{w,y}^{!}\) defines a determinantal law \(\mathsf P_{w,y}^{!}\) on \(\mathcal N_{\mathbb H}\), and \(y\mapsto\mathsf P_{w,y}^{!}\) is a weakly continuous probability kernel from \((0,\infty)\) to \(\mathcal N_{\mathbb H}\).
If \(w(y)>0\), then
\[
\widehat K_{w,y}^{!}(z,\zeta)
=
K_w(z,\zeta)
-
\frac{K_w(z,iy)K_w(iy,\zeta)}{K_w(iy,iy)}.
\]
Moreover, \(\mathsf P_{w,y}^{!}\) is a reduced Palm law of \(M_w\) at \(iy\) for \(\nu_w\)-almost every \(y\); the definition of \(\widehat K_{w,y}^{!}\) fixes the pointwise version used below for every \(y>0\).

\item[\textup{(ii)}]
If \(\tau_x\) denotes horizontal translation by \(x\), then \(\{(\tau_x)_\#\mathsf P_{w,y}^{!}\}_{x\in\mathbb R,\,y>0}\) is horizontally covariant and forms a version of the reduced Palm probability kernel of \(M_w\) relative to its intensity measure \(K_w(z,z)\,\mathrm dA(z)\).
We write \(\mathbb E_{M_w}^{x+iy,!}\) for expectation with respect to \((\tau_x)_\#\mathsf P_{w,y}^{!}\).
\end{enumerate}
\end{proposition}

\begin{proof}
We first establish existence of the determinantal laws.
The conformal covariance and projection property of the Bergman kernel (see, e.g., \cite[Chapters~15--16]{Bell1992} and \cite[Sections~4.3.10 and~5.4.1]{HoughKrishnapurPeresVirag2009}) identify \(K_{\mathbb H}\) as the kernel of the orthogonal projection from \(L^2(\mathbb H,\mathrm dA)\) onto \(A^2(\mathbb H)\).
Consequently, \(K_{\mathbb H}^{iy,!}\) is the kernel of the orthogonal projection onto \(\{f\in A^2(\mathbb H):f(iy)=0\}\).
Writing \(M_{\sqrt w}\) for multiplication by \(\sqrt w\), we have
\[
\widehat K_{w,y}^{!}
=
M_{\sqrt w}K_{\mathbb H}^{iy,!}M_{\sqrt w},
\qquad
K_w
=
M_{\sqrt w}K_{\mathbb H}M_{\sqrt w}.
\]
Since
\(
0\leq K_{\mathbb H}^{iy,!}\leq K_{\mathbb H}\leq I
\)
and \(0\leq M_{\sqrt w}\leq I\), it follows that
\[
0\leq\widehat K_{w,y}^{!}\leq K_w\leq I
\]
in operator order.
Since \(K_w\) is locally trace class, so is \(\widehat K_{w,y}^{!}\), and \cite[Theorem~3]{Soshnikov2000} yields the determinantal law \(\mathsf P_{w,y}^{!}\).

We next prove continuous dependence on \(y\).
For a relatively compact Borel set \(C\subset\mathbb H\), set
\[
g_y(z)
=
\sqrt{w(\operatorname{Im}z)}
\frac{K_{\mathbb H}(z,iy)}{\sqrt{K_{\mathbb H}(iy,iy)}},
\qquad z\in C.
\]
For \(g,h\in L^2(C)\), let \(g\otimes h\) denote the rank-one operator with kernel \(g(z)\overline{h(\zeta)}\).
The compression of \(\widehat K_{w,y}^{!}\) to \(C\) is the corresponding compression of \(K_w\) minus \(g_y\otimes g_y\).
Dominated convergence gives continuity of \(y\mapsto g_y\) in \(L^2(C)\) on compact subintervals of \((0,\infty)\), and
\[
\|g_{y_n}\otimes g_{y_n}-g_y\otimes g_y\|_1
\le
\bigl(\|g_{y_n}\|_2+\|g_y\|_2\bigr)
\|g_{y_n}-g_y\|_2.
\]
Hence the local compressions depend continuously on \(y\) in trace norm.
The Fredholm determinant formula for compactly supported Laplace functionals and the characterization of weak convergence by these functionals yield weak continuity of \(y\mapsto\mathsf P_{w,y}^{!}\), and hence the probability-kernel property.

We finally establish the Palm identification and horizontal covariance.
For \(w(y)>0\), the definitions give the representation in \textup{(i)}.
The family \(\{(\tau_x)_\#\mathsf P_{w,y}^{!}\}_{x\in\mathbb R,\,y>0}\) is horizontally covariant by construction.
Horizontal translation invariance of \(K_w\) and \cite[Theorem~6.5]{ShiraiTakahashi2003} identify this family as a version of the reduced Palm probability kernel of \(M_w\) relative to \(K_w(z,z)\,\mathrm dA(z)\).
The factorization \eqref{eq:marked-campbell-intensity} gives the \(\nu_w\)-almost-everywhere assertion in \textup{(i)}, while the defining kernel fixes the pointwise version specified in \textup{(i)} for every \(y>0\).
\end{proof}

The next lemma shows that each \(\mathsf P_{w,y}^{!}\) has a locally finite simple projection.

\begin{lemma}
\label{L:canonical-palm-projection}
For every \(y>0\), \(p_\#\xi\) is locally finite and simple for \(\mathsf P_{w,y}^{!}\)-almost every \(\xi\).
\end{lemma}

\begin{proof}
The diagonal of the defining kernel satisfies
\[
0\leq \widehat K_{w,y}^{!}(z,z)\leq K_w(z,z)
=
\frac{w(\operatorname{Im}z)}
     {4\pi(\operatorname{Im}z)^2}.
\]
Consequently, for every bounded Borel set \(B\subset\mathbb R\),
\[
\mathbb E_{\mathsf P_{w,y}^{!}}\bigl[p_\#\xi(B)\bigr]
=
\int_B\int_0^\infty
\widehat K_{w,y}^{!}(x+iv,x+iv)
\,\mathrm dv\,\mathrm dx
\le
|B|\lambda_w<\infty.
\]
Hence \(p_\#\xi([-n,n])<\infty\) almost surely for every \(n\ge1\) on a common probability-one event, proving local finiteness.

For a bounded interval \(B\subset\mathbb R\) and \(\varepsilon>0\), let \(N_{B,\varepsilon}\) be the number of ordered pairs of distinct points \(z,\zeta\in\xi\) satisfying \(\operatorname{Re}z\in B\) and \(|\operatorname{Re}z-\operatorname{Re}\zeta|<\varepsilon\).
The second factorial moment formula, Hadamard's inequality, and the diagonal bound give
\[
\mathbb E_{\mathsf P_{w,y}^{!}}N_{B,\varepsilon}
\le2\varepsilon|B|\lambda_w^2.
\]
A multiple point of \(p_\#\xi\) in \(B\) implies \(N_{B,\varepsilon}\ge1\) for every \(\varepsilon>0\), so its probability is at most \(2\varepsilon|B|\lambda_w^2\).
Letting \(\varepsilon\downarrow0\) and taking \(B=[-n,n]\), \(n\ge1\), proves simplicity.
\end{proof}

\begin{proof}[Proof of Theorem~\ref{T:projected-palm-disintegration}]
Let \(B\subset\mathbb R\) be a bounded Borel set with \(|B|>0\).
By linearity, it suffices to consider bounded measurable \(F:\mathcal N_{\mathbb R}\to[0,\infty)\).
Extend \(F\circ p_\#\) measurably to \(\mathcal N_{\mathbb H}\) by setting it to zero whenever \(p_\#\eta\notin\mathcal N_{\mathbb R}\); measurability follows from the usual counting-map description of the vague Borel \(\sigma\)-fields.
By Lemmas~\ref{L:projected-local-finiteness-stationarity}, \ref{L:projected-simplicity}, and \ref{L:canonical-palm-projection}, this convention agrees almost surely with the ordinary projection for \(M_w\), for \(M_w-\delta_z\) with \(z\in M_w\), and under every \(\mathsf P_{w,y}^{!}\).

Lemma~\ref{L:projected-simplicity}, \eqref{eq:marked-reduced-palm-campbell}, and \eqref{eq:marked-campbell-intensity} give
\[
\begin{aligned}
\mathbb E
\sum_{x\in X_w}
\mathbf 1_B(x)
F\left(\theta_x(X_w-\delta_x)\right)
&=
\mathbb E
\sum_{z\in M_w}
\mathbf 1_B(\operatorname{Re}z)
F\left(\theta_{\operatorname{Re}z}
p_\#(M_w-\delta_z)\right)
\\
&=
\lambda_w
\int_B\int_0^\infty
\mathbb E_{M_w}^{x+iy,!}
\left[
F\left(\theta_xp_\#M_w\right)
\right]
\nu_w(\mathrm dy)\,\mathrm dx.
\end{aligned}
\]
By Proposition~\ref{P:canonical-height-indexed-palm-family}, this expectation is independent of \(x\) and equals \(\mathbb E_{M_w}^{iy,!}[F(p_\#M_w)]\).
Hence
\[
\mathbb E
\sum_{x\in X_w}
\mathbf 1_B(x)
F\left(\theta_x(X_w-\delta_x)\right)
=
|B|\lambda_w
\int_0^\infty
\mathbb E_{M_w}^{iy,!}
\left[
F(p_\#M_w)
\right]
\nu_w(\mathrm dy).
\]
Comparing this identity with \eqref{eq:projected-reduced-palm-campbell} and dividing by \(|B|\lambda_w\) proves the theorem.
\end{proof}

By Proposition~\ref{P:canonical-height-indexed-palm-family},
\begin{equation}
\label{eq:joint-height-marked-palm-law}
\widehat{\mathsf P}_w^{0,!}(\mathrm dy,\mathrm d\xi)
=
\nu_w(\mathrm dy)\,\mathsf P_{w,y}^{!}(\mathrm d\xi)
\end{equation}
defines a probability measure on \((0,\infty)\times\mathcal N_{\mathbb H}\).
Set \(Y(y,\xi)=y\).
By Lemma~\ref{L:canonical-palm-projection}, \(p_\#\xi\in\mathcal N_{\mathbb R}\) almost surely under this measure.
Define \(X^{!}(y,\xi)=p_\#\xi\) on this full-measure set and as the zero configuration elsewhere.
The next corollary identifies the marginal and conditional laws under \(\widehat{\mathsf P}_w^{0,!}\).

\begin{corollary}
\label{C:projected-palm-height-law}
Under \(\widehat{\mathsf P}_w^{0,!}\), \(Y\) has distribution \(\nu_w\) given by \eqref{eq:palm-height-mixing-measure}, the kernel \(y\mapsto\mathsf P_{w,y}^{!}\) is a version of the conditional law of \(\xi\) given \(Y=y\), and \(X^{!}\) has distribution \(\mathbb P_w^{0,!}\).
\end{corollary}

\begin{proof}
The assertions concerning \(Y\) and the conditional law of \(\xi\) follow from \eqref{eq:joint-height-marked-palm-law}.
Theorem~\ref{T:projected-palm-disintegration} identifies the law of \(X^{!}\) with \(\mathbb P_w^{0,!}\).
\end{proof}

\begin{remark}[Projection and Palm conditioning]
\label{R:fixed-hidden-height-determinantal}
For each \(y>0\), \(\mathsf P_{w,y}^{!}\) is determinantal on \(\mathbb H\), but its horizontal projection is not asserted to be determinantal.
Theorem~\ref{T:projected-palm-disintegration} is a Campbell--Mecke disintegration, not ordinary conditioning on the null event \(\{X_w(\{0\})\ge1\}\), and does not assert that \(Y\) is a measurable function of \(X^{!}\).
\end{remark}

\subsection{Conditional height distributions}
\label{subsec:projected-palm-intensity}

The preceding disintegration couples the hidden Palm height with the projected reduced Palm configuration.
We derive factorial moment measures for the projected reduced Palm law from \(\rho_w^{(k)}\), and then use the marked second factorial Campbell measure to identify the conditional joint law of two height marks given their horizontal displacement.

\begin{corollary}
\label{C:projected-palm-intensity}
The intensity measure of the projected reduced Palm law is absolutely continuous, with density \(\lambda_w^{0,!}(r)=\rho_w^{(2)}(r)/\lambda_w\).
\end{corollary}

\begin{proof}
Let \(A,B\subset\mathbb R\) be bounded Borel sets with \(|B|>0\).
By \eqref{eq:projected-reduced-palm-campbell}, Lemma~\ref{L:projected-simplicity}, and the factorial moment formula in Proposition~\ref{P:projected-factorial-intensities},
\[
\begin{aligned}
|B|\lambda_w\,\mathbb E_w^{0,!}X_w(A)
&=
\mathbb E
\sum_{x_0\in X_w}
\mathbf 1_B(x_0)
\left(\theta_{x_0}(X_w-\delta_{x_0})\right)(A)
\\
&=
\mathbb E
\sum_{x_0,x_1\in X_w}^{\ne}
\mathbf 1_B(x_0)\mathbf 1_A(x_1-x_0)
=
|B|\int_A\rho_w^{(2)}(r)\,\mathrm dr.
\end{aligned}
\]
Dividing by \(|B|\lambda_w\) proves the stated density formula.
\end{proof}

\begin{lemma}
\label{L:projected-palm-factorial-moments}
Let \(q\ge1\) and let \(f:\mathbb R^q\to[0,\infty]\) be measurable.
Then
\[
\mathbb E_w^{0,!}
\sum_{u_1,\ldots,u_q\in X_w}^{\ne}
f(u_1,\ldots,u_q)
=
\frac{1}{\lambda_w}
\int_{\mathbb R^q}
f(r_1,\ldots,r_q)
\rho_w^{(q+1)}(0,r_1,\ldots,r_q)
\,\mathrm dr_1\cdots\mathrm dr_q,
\]
where \(\rho_w^{(q+1)}\) is the continuous \((q+1)\)-point factorial moment density of \(X_w\) specified in Proposition~\ref{P:projected-factorial-intensities}.
\end{lemma}

\begin{proof}
It suffices first to consider bounded compactly supported \(f\).
For a bounded Borel set \(B\subset\mathbb R\) with \(|B|>0\), apply \eqref{eq:projected-reduced-palm-campbell} to the level-\(n\) truncations of the factorial-sum functional in the statement.
Monotone convergence, Lemma~\ref{L:projected-simplicity}, and diagonal translation invariance in Proposition~\ref{P:projected-factorial-intensities}, followed by \(x_0=x\) and \(x_j=x+r_j\), give
\[
\begin{aligned}
|B|\lambda_w\,
\mathbb E_w^{0,!}
\sum_{u_1,\ldots,u_q\in X_w}^{\ne}
f(u_1,\ldots,u_q)
&=
\mathbb E
\sum_{x_0,\ldots,x_q\in X_w}^{\ne}
\mathbf 1_B(x_0)
f(x_1-x_0,\ldots,x_q-x_0)
\\
&=
|B|\int_{\mathbb R^q}
f(r_1,\ldots,r_q)
\rho_w^{(q+1)}(0,r_1,\ldots,r_q)
\,\mathrm dr_1\cdots\mathrm dr_q.
\end{aligned}
\]
Division by \(|B|\lambda_w\) proves the formula for bounded compactly supported \(f\).
Applying it to \(f_m=(f\wedge m)\mathbf 1_{[-m,m]^q}\) and letting \(m\to\infty\) proves the general case.
\end{proof}

\paragraph{Conditional kernel for a pair.}
\label{subsec:hidden-two-mark-law}

The second factorial Campbell measure admits a canonical disintegration over the horizontal displacement.

\begin{proof}[Proof of Proposition~\ref{P:hidden-two-mark-law}]
For \(r\in\mathbb R\), define
\[
\Gamma_{w,r}(\mathrm dy_0\,\mathrm dy_1)
=
w(y_0)w(y_1)
\rho_{\mathbb H}^{(2)}(iy_0,r+iy_1)
\,\mathrm dy_0\,\mathrm dy_1.
\]
The two-point determinantal formula and \eqref{eq:hardy-szego-kernel} give
\[
\rho_{\mathbb H}^{(2)}(iy_0,r+iy_1)
=
\frac{1}{16\pi^2y_0^2y_1^2}
-
\frac{1}{\pi^2\bigl(r^2+(y_0+y_1)^2\bigr)^2}.
\]
By Proposition~\ref{P:projected-factorial-intensities}, the total mass of \(\Gamma_{w,r}\) is \(\rho_w^{(2)}(r)\), which is positive by Proposition~\ref{P:projected-covariance-transform}.
Thus \(\mathsf Q_{w,r}=\Gamma_{w,r}/\rho_w^{(2)}(r)\) is a probability measure.

The factorial moment formula in Lemma~\ref{L:height-thinning-kernel} and the change of variables \(x=\operatorname{Re}z_0\) and \(r=\operatorname{Re}z_1-\operatorname{Re}z_0\) give
\[
\begin{aligned}
&\mathbb E\sum_{z_0,z_1\in M_w}^{\ne}
H\bigl(\operatorname{Re}z_0,
       \operatorname{Re}z_1-\operatorname{Re}z_0,
       \operatorname{Im}z_0,\operatorname{Im}z_1\bigr)
\\
&\quad=
\int_{\mathbb R}\int_{\mathbb R}
\int_{(0,\infty)^2}
H(x,r,y_0,y_1)
\,\Gamma_{w,r}(\mathrm dy_0\,\mathrm dy_1)
\,\mathrm dr\,\mathrm dx.
\end{aligned}
\]
Substituting \(\Gamma_{w,r}=\rho_w^{(2)}(r)\mathsf Q_{w,r}\) gives the asserted disintegration.

If \(r_n\to r\), the densities of \(\Gamma_{w,r_n}\) converge pointwise to that of \(\Gamma_{w,r}\).
They are supported on \([\alpha,\beta]^2\) and, by \eqref{eq:hardy-szego-hadamard-bound}, are uniformly bounded there by \((16\pi^2\alpha^4)^{-1}\).
Dominated convergence therefore gives
\(
\|\Gamma_{w,r_n}-\Gamma_{w,r}\|_{\mathrm{TV}}\longrightarrow0.
\)
Taking total masses and then normalizing yields
\[
\|\mathsf Q_{w,r_n}-\mathsf Q_{w,r}\|_{\mathrm{TV}}\longrightarrow0.
\]

Almost-everywhere uniqueness of disintegration kernels implies that any other weakly continuous version of the conditional height kernel agrees with \(\mathsf Q_{w,r}\) for \(\rho_w^{(2)}(r)\,\mathrm dr\)-almost every \(r\).
Since \(\rho_w^{(2)}(r)>0\) for every \(r\), the agreement set is dense, and weak continuity extends the equality to every \(r\), proving uniqueness.
\end{proof}

The conditional height kernel in Proposition~\ref{P:hidden-two-mark-law} has the following limiting laws at vanishing and diverging horizontal separation.

\begin{corollary}
\label{C:far-hidden-height-limit}
In total variation,
\[
\mathsf Q_{w,r}\longrightarrow\nu_w\otimes\nu_w
\quad\text{as }|r|\to\infty,
\qquad
\mathsf Q_{w,r}\longrightarrow\mathsf Q_{w,0}
\quad\text{as }r\to0.
\]
The measure \(\mathsf Q_{w,0}\) has density
\begin{equation}
\label{eq:zero-displacement-height-density}
\frac{w(y_0)w(y_1)}{\rho_w^{(2)}(0)}
\left[
\frac{1}{16\pi^2y_0^2y_1^2}
-
\frac{1}{\pi^2(y_0+y_1)^4}
\right]
\end{equation}
with respect to \(\mathrm dy_0\,\mathrm dy_1\).
\end{corollary}

\begin{proof}
As \(|r|\to\infty\), dominated convergence, \eqref{eq:palm-height-mixing-measure}, and \eqref{eq:projected-intensity} give
\[
\rho_w^{(2)}(r)\mathsf Q_{w,r}
\longrightarrow
\frac{w(y_0)w(y_1)}
     {16\pi^2y_0^2y_1^2}
\,\mathrm dy_0\,\mathrm dy_1
=
\lambda_w^2(\nu_w\otimes\nu_w)
\]
in total variation.
Taking total masses gives \(\rho_w^{(2)}(r)\to\lambda_w^2\), and normalization yields the first convergence.
The second follows from the total-variation continuity in Proposition~\ref{P:hidden-two-mark-law}; setting \(r=0\) gives the stated density.
\end{proof}

The bracketed factor in \eqref{eq:zero-displacement-height-density} admits the factorization
\[
\frac{1}{16\pi^2y_0^2y_1^2}
-
\frac{1}{\pi^2(y_0+y_1)^4}
=
\frac{(y_0-y_1)^2(y_0^2+6y_0y_1+y_1^2)}
     {16\pi^2y_0^2y_1^2(y_0+y_1)^4}.
\]
Consequently, the density of \(\mathsf Q_{w,0}\) vanishes precisely when \(y_0=y_1\) or \(w(y_0)w(y_1)=0\).
Here \(\mathsf Q_{w,0}\) is the value at \(r=0\) of the total-variation continuous version fixed by Proposition~\ref{P:hidden-two-mark-law}.

\subsection{Right nearest-neighbour small-spacing asymptotic}
\label{subsec:linear-right-short-gap}

Recall that, under \(\mathbb P_w^{0,!}\),
\[
S_w
=
\inf\bigl(\operatorname{supp}X_w\cap(0,\infty)\bigr),
\qquad
\inf\varnothing:=\infty.
\]
By Corollary~\ref{C:projected-palm-intensity}, the reduced Palm configuration has no point at any fixed location, so the choice of open or closed endpoints below is immaterial.

\begin{proof}[Proof of Theorem~\ref{T:linear-right-short-gap}]
For \(t>0\), set \(N_t=X_w((0,t))\).
Then \(\mathbb P_w^{0,!}\{S_w\le t\}=\mathbb P_w^{0,!}\{N_t\ge1\}\), and the first two Bonferroni inequalities give
\[
\mathbb E_w^{0,!}N_t
-\frac12\mathbb E_w^{0,!}[N_t(N_t-1)]
\le
\mathbb P_w^{0,!}\{S_w\le t\}
\le
\mathbb E_w^{0,!}N_t.
\]

By \eqref{eq:projected-two-point-density}, \(\rho_w^{(2)}\) is even.
Since \(y+v\ge2\alpha\) on the support of \(w(y)w(v)\), differentiation under the integral sign gives \(\rho_w^{(2)}\in C^2\) near \(0\).
Hence
\[
\rho_w^{(2)}(r)
=
\rho_w^{(2)}(0)+O_w(r^2).
\]
By Corollary~\ref{C:projected-palm-intensity},
\[
\mathbb E_w^{0,!}N_t
=
\frac{1}{\lambda_w}\int_0^t\rho_w^{(2)}(r)\,\mathrm dr
=
\frac{\rho_w^{(2)}(0)}{\lambda_w}\,t+O_w(t^3).
\]

Lemma~\ref{L:projected-palm-factorial-moments}, applied with \(q=2\) and \(f=\mathbf 1_{(0,t)^2}\), and the bound \(0\le\rho_w^{(3)}\le\lambda_w^3\) from Proposition~\ref{P:projected-factorial-intensities} give
\[
\mathbb E_w^{0,!}[N_t(N_t-1)]
=
\frac{1}{\lambda_w}
\int_0^t\int_0^t
\rho_w^{(3)}(0,r,s)
\,\mathrm dr\,\mathrm ds
\le
\lambda_w^2t^2.
\]
Combining these estimates with the Bonferroni inequalities gives the asserted asymptotic.
Its coefficient is strictly positive because \(\lambda_w>0\) by \eqref{eq:projected-intensity} and \(\rho_w^{(2)}(0)>0\) by Proposition~\ref{P:projected-covariance-transform}.
\end{proof}

\section{Identifiability and stable finite-union reconstruction}
\label{sec:stieltjes-laplace-identifiability}

This section proves Theorems~\ref{T:stieltjes-laplace-identifiability} and \ref{T:finite-union-stability}.
The former establishes identifiability for arbitrary finite positive measures compactly supported in \((0,\infty)\), while the latter gives Lipschitz stability on separated compact classes of finite-union profiles.

\subsection{Identifiability of compactly supported height measures}
\label{subsec:covariance-transform-general-measures}

For a finite positive measure \(\mu\) compactly supported in \((0,\infty)\), recall the covariance transform \(G_\mu\) defined by \eqref{eq:height-measure-covariance-transform}.
If \(\mu_w(\mathrm dy)=w(y)\,\mathrm dy\), then Proposition~\ref{P:projected-covariance-transform} gives
\[
g_w(r)=-\pi^{-2}G_{\mu_w}(r).
\]
Hence identifiability of \(\mu_w\) from \(G_{\mu_w}\) is equivalent to identifiability from the second factorial cumulant density of the projected process.

Writing \(\nu=\mu\ast\mu\), the definition of convolution gives
\begin{equation}
\label{eq:covariance-transform-convolution}
G_\mu(r)
=
\int_0^\infty
\frac{1}{(r^2+s^2)^2}
\,\nu(\mathrm ds).
\end{equation}
The inverse argument first recovers \(\nu\) from \(G_\mu\) by Stieltjes uniqueness and then recovers \(\mu\) from \(\nu\) by Laplace uniqueness, with positivity selecting the square root.

The following lemma provides the analyticity needed to extend equality of covariance transforms from an interval to all \(r>0\).

\begin{lemma}
\label{L:covariance-bilinear-analyticity}
Let \(\sigma\) and \(\tau\) be finite signed measures supported in \([a,b]\subset(0,\infty)\).
Then
\[
\mathcal B_{\sigma,\tau}(z)
=
\int_{[a,b]}\int_{[a,b]}
\frac{\sigma(\mathrm dy)\tau(\mathrm dv)}
     {\bigl(z^2+(y+v)^2\bigr)^2}
\]
is holomorphic on
\(
\{z\in\mathbb C:|\operatorname{Im}z|<2a\}.
\)
\end{lemma}

\begin{proof}
For fixed \(y,v\in[a,b]\), the integrand is holomorphic on the stated strip, since its poles are \(z=\pm i(y+v)\).
Moreover, for every compact subset \(K\) of the strip,
\[
c_K
:=
\inf_{\substack{z\in K\\s\in[2a,2b]}}
|z^2+s^2|
>0.
\]
Thus the integrand is bounded by \(c_K^{-2}\), which is integrable with respect to the finite measure \(|\sigma|\otimes|\tau|\).
The same holds for the \(z\)-derivative, so differentiation under the integral sign proves the claim.
\end{proof}

We record the Stieltjes and Laplace transform uniqueness statements needed below for compactly supported measures.
The corresponding results for nonnegative measures are given in \cite[Proposition~1.2 and the discussion following Definition~2.1]{SchillingSongVondracek2012}; see also \cite[Chapter~II, Theorem~6.3, and Chapter~VIII, Theorem~5b]{Widder1941} for the classical formulations in terms of functions of bounded variation.
The extension from nonnegative to finite signed measures follows immediately from the Jordan decomposition, so we state the results without proof.

\begin{lemma}[Stieltjes transform uniqueness]
\label{L:stieltjes-uniqueness-compact}
Let \(\sigma\) be a finite signed measure supported in a compact subset of \([0,\infty)\).
If \(\int_0^\infty(t+q)^{-1}\,\sigma(\mathrm dq)=0\) for every \(t>0\), then \(\sigma=0\).
\end{lemma}

\begin{lemma}[Laplace transform uniqueness]
\label{L:laplace-uniqueness-compact}
Let \(\sigma\) be a finite signed measure supported in a compact subset of \([0,\infty)\).
If \(\int_0^\infty e^{-\lambda y}\,\sigma(\mathrm dy)=0\) for every \(\lambda>0\), then \(\sigma=0\).
\end{lemma}

\begin{proof}[Proof of Theorem~\ref{T:stieltjes-laplace-identifiability}]
Choose \(0<a<b<\infty\) such that both measures are supported in \([a,b]\).
By Lemma~\ref{L:covariance-bilinear-analyticity}, \(G_\mu\) and \(G_{\widetilde\mu}\) are real analytic on \(\mathbb R\).
Their equality on a nonempty open subinterval therefore extends to all \(r>0\).

Set \(\nu=\mu\ast\mu\) and \(\widetilde\nu=\widetilde\mu\ast\widetilde\mu\).
These are finite positive measures supported in \([2a,2b]\).
For \(t>0\), define \(H_\mu(t)=\int_t^\infty G_\mu(\sqrt{u})\,\mathrm du\).
By \eqref{eq:covariance-transform-convolution} and Tonelli's theorem,
\[
H_\mu(t)
=
\int_0^\infty
\left(
\int_t^\infty\frac{\mathrm du}{(u+s^2)^2}
\right)
\nu(\mathrm ds)
=
\int_0^\infty\frac{\nu(\mathrm ds)}{t+s^2}.
\]
The analogous identity holds for \(\widetilde\mu\), and hence \(H_\mu=H_{\widetilde\mu}\) on \((0,\infty)\).

Let \(\eta\) and \(\widetilde\eta\) be the pushforwards of \(\nu\) and \(\widetilde\nu\), respectively, under \(s\mapsto s^2\).
These are finite positive measures supported in \([4a^2,4b^2]\), and
\[
\int_0^\infty
\frac{(\eta-\widetilde\eta)(\mathrm dq)}{t+q}
=
0,
\qquad t>0.
\]
Lemma~\ref{L:stieltjes-uniqueness-compact} gives \(\eta=\widetilde\eta\).
Since \(s\mapsto s^2\) is a bijection of \([0,\infty)\), it follows that \(\nu=\widetilde\nu\), and hence \(\mu\ast\mu=\widetilde\mu\ast\widetilde\mu\).

Taking Laplace transforms yields
\[
\left(
\int_0^\infty e^{-z y}\,\mu(\mathrm dy)
\right)^2
=
\left(
\int_0^\infty e^{-z y}\,\widetilde\mu(\mathrm dy)
\right)^2,
\qquad z>0.
\]
Since both integrals are nonnegative, they are equal.
Lemma~\ref{L:laplace-uniqueness-compact}, applied to \(\mu-\widetilde\mu\), therefore gives \(\mu=\widetilde\mu\).
\end{proof}

\subsection{Finite-union profiles and the covariance map}
\label{subsec:finite-union-exponential-sum}

For \(m\ge1\), recall that a finite-union profile is parametrized by
\[
I(\vartheta)
=
\bigcup_{j=1}^{m}[a_j,b_j],
\qquad
0<a_1<b_1<a_2<b_2<\cdots<a_m<b_m<\infty,
\]
where \(\vartheta=(a_1,b_1,\ldots,a_m,b_m)\in\mathbb R^{2m}\) and \(\mu_\vartheta(\mathrm dy)=\mathbf 1_{I(\vartheta)}(y)\,\mathrm dy\).
The identifiability theorem gives exact reconstruction within this class.

\begin{proposition}
\label{P:finite-union-exponential-sum}
The covariance transform \(G_{\mu_\vartheta}\) determines \(\mu_\vartheta\), and hence determines \(m\) and all endpoints \(a_j,b_j\).
Moreover,
\[
t\int_0^\infty e^{-t y}\,\mu_\vartheta(\mathrm dy)
=
\sum_{j=1}^{m}\bigl(e^{-a_jt}-e^{-b_jt}\bigr),
\qquad t>0.
\]
\end{proposition}

\begin{proof}
By Theorem~\ref{T:stieltjes-laplace-identifiability}, \(G_{\mu_\vartheta}\) determines \(\mu_\vartheta\).
Since \(\operatorname{supp}\mu_\vartheta=I(\vartheta)\), its connected components determine \(m\) and all their endpoints.
The displayed identity follows by integrating \(e^{-t y}\) over the component intervals.
\end{proof}

For the stability analysis, fix \(m\geq1\), \(0<\alpha<\beta<\infty\), and \(\delta>0\), and recall the separated compact class
\[
\Theta_{m,\alpha,\beta,\delta}
=
\left\{
\vartheta\in\mathbb R^{2m}:
\begin{array}{l}
\alpha\leq a_1,\quad b_m\leq\beta,\\
b_j-a_j\geq\delta\quad (1\leq j\leq m),\\
a_{j+1}-b_j\geq\delta\quad (1\leq j<m)
\end{array}
\right\}.
\]
This is a compact convex subset of \(\mathbb R^{2m}\), and it is nonempty if and only if
\[
(2m-1)\delta\leq\beta-\alpha.
\]
Indeed, the \(m\) interval lengths and \(m-1\) intervening gaps require a total length of at least \((2m-1)\delta\); conversely, an admissible configuration is obtained by taking all these lengths equal to \(\delta\).

Fix \(0<r_-<r_+<\infty\) and equip \(C([r_-,r_+])\) with the supremum norm.
Define the covariance map by
\[
\mathcal G_m:
\Theta_{m,\alpha,\beta,\delta}
\longrightarrow C([r_-,r_+]),
\qquad
\mathcal G_m(\vartheta)
=
G_{\mu_\vartheta}\big|_{[r_-,r_+]}.
\]
To differentiate this map, introduce the open neighborhood
\[
U_m
=
\left\{
\vartheta\in\mathbb R^{2m}:
\begin{array}{l}
\alpha/2<a_1,\quad b_m<2\beta,\\
b_j-a_j>\delta/2\quad (1\leq j\leq m),\\
a_{j+1}-b_j>\delta/2\quad (1\leq j<m)
\end{array}
\right\}
\]
of \(\Theta_{m,\alpha,\beta,\delta}\), and extend \(\mathcal G_m\) to
\[
\widetilde{\mathcal G}_m:
U_m\longrightarrow C([r_-,r_+]),
\qquad
\widetilde{\mathcal G}_m(\vartheta)
=
G_{\mu_\vartheta}\big|_{[r_-,r_+]}.
\]

The following lemma establishes the smoothness and injectivity properties of the covariance map needed for stability.

\begin{lemma}
\label{L:finite-union-immersion}
The map \(\widetilde{\mathcal G}_m:U_m\to C([r_-,r_+])\) is \(C^\infty\), and its restriction \(\mathcal G_m\) to \(\Theta_{m,\alpha,\beta,\delta}\) is injective.
Moreover, for every \(\vartheta\in\Theta_{m,\alpha,\beta,\delta}\), the derivative
\[
\mathrm D\widetilde{\mathcal G}_m(\vartheta):
\mathbb R^{2m}\longrightarrow C([r_-,r_+])
\]
is injective.
\end{lemma}

\begin{proof}
Fix \(\vartheta=(a_1,b_1,\ldots,a_m,b_m)\in U_m\), and set
\[
\ell_j:=b_j-a_j,
\qquad
\phi_j(u):=a_j+u\ell_j,
\qquad
u\in[0,1],\quad 1\leq j\leq m.
\]
The changes of variables \(y=\phi_j(s)\) and \(v=\phi_k(t)\) give, for \(r\in[r_-,r_+]\),
\[
\widetilde{\mathcal G}_m(\vartheta)(r)
=
\sum_{j,k=1}^m
\ell_j\ell_k
\int_{[0,1]^2}
\frac{\mathrm ds\,\mathrm dt}
{\bigl(r^2+(\phi_j(s)+\phi_k(t))^2\bigr)^2}.
\]
For every compact subset of \(U_m\), all endpoint derivatives of the summands are uniformly bounded on \([r_-,r_+]\times[0,1]^2\).
Differentiation under the integrals, uniformly in \(r\), therefore gives \(\widetilde{\mathcal G}_m\in C^\infty(U_m;C([r_-,r_+]))\).

If \(\vartheta,\widetilde\vartheta\in \Theta_{m,\alpha,\beta,\delta}\) and \(\mathcal G_m(\vartheta)=\mathcal G_m(\widetilde\vartheta)\), then Theorem~\ref{T:stieltjes-laplace-identifiability}, applied to this equality on \((r_-,r_+)\), gives \(\mu_\vartheta=\mu_{\widetilde\vartheta}\).
The ordered connected components of their common support have the same endpoints, so \(\vartheta=\widetilde\vartheta\).
Thus \(\mathcal G_m\) is injective.

For \(\vartheta\in U_m\) and \(h=(h_1,\ldots,h_{2m})\in\mathbb R^{2m}\), set
\[
\dot\mu_{\vartheta,h}
=
\sum_{j=1}^{m}
\left(
-h_{2j-1}\delta_{a_j}+h_{2j}\delta_{b_j}
\right).
\]
Differentiating the component intervals and using the symmetry of the kernel gives, for \(r\in[r_-,r_+]\),
\[
\mathrm D\widetilde{\mathcal G}_m(\vartheta)[h](r)
=
2
\int_0^\infty\int_0^\infty
\frac{
\dot\mu_{\vartheta,h}(\mathrm dy)\,
\mu_\vartheta(\mathrm dv)}
{\bigl(r^2+(y+v)^2\bigr)^2}.
\]

Now let \(\vartheta\in\Theta_{m,\alpha,\beta,\delta}\) and assume that \(\mathrm D\widetilde{\mathcal G}_m(\vartheta)[h]=0\).
By Lemma~\ref{L:covariance-bilinear-analyticity}, this equality extends to all \(r>0\).
For the finite signed compactly supported measure \(\sigma=\dot\mu_{\vartheta,h}\ast\mu_\vartheta\), it follows that \(\int_0^\infty(r^2+s^2)^{-2}\,\sigma(\mathrm ds)=0\) for every \(r>0\).
Since \(\sigma\) is finite and \(\int_t^\infty(u+s^2)^{-2}\,\mathrm du\leq t^{-1}\), Fubini's theorem applies, and integration with respect to \(u=r^2\) gives
\[
\int_0^\infty
\frac{\sigma(\mathrm ds)}{t+s^2}
=
0,
\qquad t>0.
\]
Let \(\eta\) be the pushforward of \(\sigma\) under \(s\mapsto s^2\).
Lemma~\ref{L:stieltjes-uniqueness-compact} gives \(\eta=0\), and the bijectivity of \(s\mapsto s^2\) on \([0,\infty)\) gives \(\sigma=0\).

Taking Laplace transforms of \(\dot\mu_{\vartheta,h}\ast\mu_\vartheta=0\) gives
\[
\left(
\int_0^\infty
e^{-z y}\,\dot\mu_{\vartheta,h}(\mathrm dy)
\right)
\left(
\int_0^\infty
e^{-z y}\,\mu_\vartheta(\mathrm dy)
\right)
=
0,
\qquad z>0.
\]
The second factor is strictly positive, so the first vanishes for every \(z>0\).
Lemma~\ref{L:laplace-uniqueness-compact} gives \(\dot\mu_{\vartheta,h}=0\).
Since the endpoint atoms are pairwise distinct, all their coefficients vanish, and hence \(h=0\).
\end{proof}

\subsection{Stable recovery}
\label{subsec:fixed-m-stability}

\begin{proof}[Proof of Theorem~\ref{T:finite-union-stability}]
Write
\[
\Theta:=\Theta_{m,\alpha,\beta,\delta},
\qquad
\|\cdot\|_2:=\|\cdot\|_{\mathbb R^{2m}},
\qquad
\|\cdot\|_C:=\|\cdot\|_{C([r_-,r_+])}.
\]
By Lemma~\ref{L:finite-union-immersion}, \(\mathrm D\widetilde{\mathcal G}_m(\vartheta)\) is injective for every \(\vartheta\in\Theta\).
Continuity and compactness therefore give
\[
\eta
:=
\min_{\substack{\vartheta\in\Theta\\ \|h\|_2=1}}
\left\|
\mathrm D\widetilde{\mathcal G}_m(\vartheta)[h]
\right\|_C
>0.
\]

Since \(\widetilde{\mathcal G}_m\) is \(C^2\) on \(U_m\) and \(\Theta\) is compact, its second derivative is uniformly bounded on \(\Theta\).
Convexity of \(\Theta\) and Taylor's theorem consequently give
\[
\left\|
\mathcal G_m(\vartheta+h)-\mathcal G_m(\vartheta)
-
\mathrm D\widetilde{\mathcal G}_m(\vartheta)[h]
\right\|_C
\leq
M_2\|h\|_2^2
\]
for some \(M_2<\infty\) whenever \(\vartheta,\vartheta+h\in\Theta\).

Choose \(\varepsilon_0>0\) such that \(M_2\varepsilon_0\leq\eta/2\).
For \(\vartheta,\widetilde\vartheta\in\Theta\), set \(h=\widetilde\vartheta-\vartheta\).
If \(\|h\|_2\leq\varepsilon_0\), then
\[
\left\|
\mathcal G_m(\widetilde\vartheta)-\mathcal G_m(\vartheta)
\right\|_C
\geq
\eta\|h\|_2-M_2\|h\|_2^2
\geq
\frac{\eta}{2}\|h\|_2.
\]

For the remaining pairs, let
\[
K_{\varepsilon_0}
=
\left\{
(\vartheta,\widetilde\vartheta)\in\Theta^2:
\|\vartheta-\widetilde\vartheta\|_2\geq\varepsilon_0
\right\}.
\]
If \(K_{\varepsilon_0}=\varnothing\), the preceding estimate proves the theorem.
Otherwise, compactness of \(K_{\varepsilon_0}\), continuity of \(\mathcal G_m\), and injectivity of its restriction to \(\Theta\) give
\[
d_0
:=
\min_{(\vartheta,\widetilde\vartheta)\in K_{\varepsilon_0}}
\left\|
\mathcal G_m(\vartheta)-\mathcal G_m(\widetilde\vartheta)
\right\|_C
>
0.
\]
Writing \(D:=\operatorname{diam}_{\|\cdot\|_2}\Theta\), every pair in \(K_{\varepsilon_0}\) satisfies
\[
\|\vartheta-\widetilde\vartheta\|_2
\leq
D
\leq
\frac{D}{d_0}
\left\|
\mathcal G_m(\vartheta)-\mathcal G_m(\widetilde\vartheta)
\right\|_C.
\]
Combining the two estimates proves the theorem, with \(C=\max\{2/\eta,D/d_0\}\) when \(K_{\varepsilon_0}\neq\varnothing\) and \(C=2/\eta\) otherwise.
In either case, \(C\) depends only on \(m,\alpha,\beta,\delta,r_-\), and \(r_+\).
\end{proof}

We now extend the fixed-\(m\) stability estimate to a bounded but unknown number of component intervals.
Fix \(M\geq1\), and let
\[
\mathfrak M
:=
\left\{
1\leq m\leq M:
(2m-1)\delta\leq\beta-\alpha
\right\}.
\]
Assume that \(\mathfrak M\neq\varnothing\), and set
\[
\Theta_{\leq M}
:=
\bigsqcup_{m\in\mathfrak M}
\Theta_{m,\alpha,\beta,\delta}.
\]

\begin{corollary}
\label{C:stable-recovery-of-m}
For distinct \(m,\ell\in\mathfrak M\),
\[
\operatorname{dist}_{C([r_-,r_+])}
\left(
\mathcal G_m(\Theta_{m,\alpha,\beta,\delta}),
\mathcal G_\ell(\Theta_{\ell,\alpha,\beta,\delta})
\right)
>
0.
\]
Moreover, there exist constants \(\varepsilon,C>0\), depending only on \(M,\alpha,\beta,\delta,r_-\), and \(r_+\), such that whenever \(m,\ell\in\mathfrak M\), \(\vartheta\in\Theta_{m,\alpha,\beta,\delta}\), and \(\widetilde\vartheta\in\Theta_{\ell,\alpha,\beta,\delta}\) satisfy
\[
\left\|
\mathcal G_m(\vartheta)-\mathcal G_\ell(\widetilde\vartheta)
\right\|_{C([r_-,r_+])}
<
\varepsilon,
\]
one has \(m=\ell\) and
\[
\|\vartheta-\widetilde\vartheta\|_{\mathbb R^{2m}}
\leq
C
\left\|
\mathcal G_m(\vartheta)-\mathcal G_m(\widetilde\vartheta)
\right\|_{C([r_-,r_+])}.
\]
In particular, \(G_{\mu_\vartheta}\big|_{[r_-,r_+]}\) determines the number of intervals and all endpoints on \(\Theta_{\leq M}\).
\end{corollary}

\begin{proof}
For \(m\in\mathfrak M\), set \(\mathcal A_m:=\mathcal G_m(\Theta_{m,\alpha,\beta,\delta})\).
This set is compact by the compactness of \(\Theta_{m,\alpha,\beta,\delta}\) and the continuity of \(\mathcal G_m\).
If \(\mathcal A_m\cap\mathcal A_\ell\neq\varnothing\), then two corresponding covariance transforms agree on \((r_-,r_+)\).
Theorem~\ref{T:stieltjes-laplace-identifiability} gives equality of the associated height measures and hence of their supports, forcing \(m=\ell\).
Thus the sets \(\mathcal A_m\) are pairwise disjoint, and distinct such sets have positive distance because they are compact.

If \(\mathfrak M\) contains at least two elements, choose \(\varepsilon\) to be the minimum of these finitely many positive distances; otherwise, set \(\varepsilon=1\).
Transforms at distance less than \(\varepsilon\) must then correspond to the same value of \(m\).

For each \(m\in\mathfrak M\), let \(C_m\) be the constant in Theorem~\ref{T:finite-union-stability}.
Taking \(C:=\max_{m\in\mathfrak M}C_m\) gives the endpoint estimate.
The final assertion follows by taking the covariance transforms equal.
\end{proof}

\begin{remark}
\label{R:why-separated-classes-are-needed}
The lower bounds on interval lengths and gaps are essential for uniform recovery.
If an interval length or gap approaches zero, a component may disappear or two components may merge while the corresponding covariance transforms remain arbitrarily close, so different component counts cannot be uniformly separated.
\end{remark}

\end{document}